\documentclass[8pt]{elsarticle}

\usepackage[english]{babel}
\usepackage{graphicx,epstopdf,epsfig}
\usepackage{amsfonts,epsfig,fancyhdr,graphics,amsmath,amssymb,float}
\usepackage{enumerate,mathtools,todonotes}
\usepackage[matrix,arrow,tips,curve]{xy}
\usepackage{latexsym,eucal,color}
\usepackage{pstricks}
\usepackage{dsfont}
\usepackage{setspace}
\usepackage{multicol}
\usepackage{enumitem}
\usepackage{tikz,pgfplots,graphicx}
\usetikzlibrary{patterns}
\usetikzlibrary{shapes,arrows,positioning,decorations.markings,math}
\usepackage{tkz-euclide}
\pgfplotsset{compat=1.11}
\usepackage{hyperref}
\hypersetup{colorlinks=true,
    linkcolor=blue,
    citecolor=blue,
    filecolor=cyan,      
    urlcolor=magenta
}
\usepackage[capitalise]{cleveref}  

\newtheorem{theorem}{Theorem}[section]
\newtheorem{lemma}[theorem]{Lemma}
\newtheorem{corollary}[theorem]{Corollary}
\newtheorem{proposition}[theorem]{Proposition}
\usepackage{subcaption}

\newcommand{\foroverarrowvee}[1][1.1ex]{\mathrel{\overrightarrow{\rule{0pt}{1.6ex}\saj }}}
\newcommand{\backoverarrowvee}[1][1.1ex]{\mathrel{\overleftarrow{\rule{0pt}{1.6ex}\saj }}}
\newcommand{\symoverarrowvee}[1][1.1ex]{\mathrel{\overleftrightarrow{\rule{0pt}{1.6ex}\saj }}}

\newcommand{\svj}{\mathbin{\dot\vee}}
\newcommand{\saj}{\mathbin{\underline{\smash{\vee}}}}

\begin{document} 
\bibliographystyle{abbrv}
\setcounter{page}{1}

\thispagestyle{empty}

\title{Spectra of Subdivision Products of Digraphs}

\date{~} 
  
\author[MC]{M. Cavers\texorpdfstring{\corref{cor1}}{}}
\ead{michael.cavers@utoronto.ca}
\author[FM]{F. Maghsoudi}
\ead{farzad.maghsoudi93@gmail.com}
\author[BM]{B. Miraftab}
\ead{bobby.miraftab@gmail.com}

\cortext[cor1]{Corresponding author}

\address[MC]{Department of Computer and Mathematical Sciences, University of Toronto Scarborough,\\ Toronto, ON Canada M1C 1A4}
\address[BM]{School of Computer Science,
Carleton University,\\ Ottawa, Ontario, K1S 5B6, Canada}   

\begin{abstract}
This paper introduces four types of subdivision products for simple directed graphs extending those from the undirected case, in particular, the subdivision-vertex join, subdivision-arc join, subdivision-vertex corona and subdivision-arc corona.
Structural and spectral properties of these constructions are analyzed, with a focus on adjacency, Laplacian and signless Laplacian spectra.
\end{abstract}


\maketitle
\section{Introduction}
Classical graph products such as disjoint union, Cartesian, Kronecker (tensor), lexicographic, direct, and strong products have been extensively studied, together with the spectra of their adjacency, Laplacian, and signless Laplacian matrices \cite{cvetkovic1980spectra, godsil2001algebraic, imrich2000product}. 
Subdivision-based product constructions have also been considered and of particular relevance are the subdivision-vertex and subdivision-edge joins formed by first subdividing the edges of one graph and then “joining’’ either the original vertices or the newly inserted vertices to a second graph.
Their adjacency, Laplacian, and signless Laplacian spectra are  well understood when the first factor is a regular graph \cite{Indulal,Liu,shinoda1980chara}. 
Subdivision corona products of two graphs have also been studied from a spectral approach \cite{BarikKalita,barik2017laplacian,das2017normalized,fiuj2023four,LIU-subdivision,PengliMiao2013,varghese2016spectra}. 

In the directed setting, recent work develops a unified spectral framework for standard digraph products (Cartesian, lexicographic, direct, strong) and related matrix constructions \cite{catral2020spectra}. 
More recently, digraphs with few adjacency eigenvalues have been studied \cite{cavers2025digraphs}, and extensions of corona products to the directed graph setting have been introduced with a focus on their adjacency, Laplacian, and signless Laplacian spectra \cite{cavers2025spectra}.
Subdivision digraphs, obtained by subdividing every arc of a digraph, have also been considered and have played a meaningful role in spectral theory, for instance in showing that every complex number is a limit point of digraph eigenvalues \cite{zhang2006limit}. 
They also appear as a special case in a rich family of digraph transformations built from incidence and line-digraph operations \cite{DengKelmans2013}.
For further background on digraph spectral theory we refer to Brualdi’s survey \cite{Brualdi}. 

The purpose of this paper is to generalize subdivision-based operations to digraphs and determine the adjacency, Laplacian, and signless Laplacian spectra of such products in terms of the spectra of the factors, thereby extending to digraphs what is known for graphs. 
Preliminary material needed for this development is presented in \cref{sec:prelim}, while in \cref{sec:subdiv}, we review known spectral results on subdivision graphs and digraphs.
In \cref{sec:sub-join}, we introduce the subdivision-vertex join and subdivision-arc join, defined by first subdividing each arc of a digraph and then attaching a second digraph to the original vertices or the newly inserted vertices with bidirectional arcs, respectively. 
Finally, \cref{sec:subcor} introduces the subdivision-vertex corona and subdivision-arc corona of two digraphs.

\section{Preliminaries}\label[section]{sec:prelim}
All graphs and digraphs are assumed to be nonempty and simple (i.e., no loops or multiple edges/arcs), however, $2$-cycles are permitted in digraphs.
Terminology not defined in this paper that relates to graphs can be found in \cite{brouwer2011spectra, cvetkovic1980spectra, godsil2001algebraic}.
Throughout, we assume $D$ is a digraph with vertex set $\{v_1,v_2,\ldots,v_n\}$ and arc set $E(D)$ which consists of ordered pairs of vertices (called \emph{arcs}).
We write $uv\in E(D)$ to mean $(u,v)\in E(D)$. 
If both $uv$ and $vu$ belong to $E(D)$, then the arcs $uv,vu$ form a \emph{symmetric edge}. 
A digraph is \emph{strongly connected} if it is possible to reach every vertex starting from any other vertex by traversing arcs in the direction in which they point.
The \emph{underlying graph} $U(D)$ has vertex set $V(U(D))=V(D)$ and edge set $E(U(D))=\big\{\{u,v\} : uv\in E(D)\big\}$.
The \emph{adjacency matrix} $A(D)=[a_{ij}]$ of $D$ is the $n\times n$ matrix defined by
$a_{ij}=1$ if $v_iv_j\in E(D)$ and $a_{ij}=0$ otherwise.
We say that $D$ is \emph{$r$-out-regular} if every vertex has out-degree equal to $r$.
The \emph{out-degree matrix} (resp. \emph{in-degree matrix}) $D_{\textnormal{out}}(D)$ (resp. $D_{\textnormal{in}}(D)$) of $D$ is the diagonal matrix whose $i$-th diagonal entry equals the out-degree (resp. in-degree) of vertex $v_i$.
The \emph{Laplacian matrix} of $D$ is $L(D)=D_{\textnormal{out}}(D)-A(D)$
and the \emph{signless Laplacian matrix} of $D$ is $Q(D)=D_{\textnormal{out}}(D)+A(D)$. 
The \emph{characteristic polynomial} of an $n\times n$ matrix $M$ is $f_M(\lambda) = \det(\lambda I_n - M)$. 
When $M$ is the adjacency (resp.\ Laplacian or signless Laplacian) matrix of a digraph $D$, the roots of $f_M(\lambda)$ form the adjacency (resp.\ Laplacian or signless Laplacian) spectrum of $D$, and are often referred to as the $A$-spectrum (resp. $L$-spectrum or $Q$-spectrum).

The line digraph of $D$, denoted by $\mathcal{L}(D)$, has one vertex for each arc of $D$, and two vertices representing arcs from $u$ to $v$ and from $w$ to $x$ in $D$ are connected by an arc from $uv$ to $wx$ in the line digraph if and only if $v = w$.
For a digraph $D$ with $V(D)=\{v_1,\ldots,v_n\}$ and arcs ordered as $e_1,\ldots,e_m$, its \emph{in-incidence matrix} $B_{\textnormal{in}}(D)=[b_{ij}]$ is an $n\times m$ matrix such that $b_{ij}=1$ if $e_j=v_kv_i$ for some vertex $v_k$, and $b_{ij}=0$ otherwise.
Similarly, the \emph{out-incidence matrix} of $D$ is the $n\times m$ matrix  $B_{\textnormal{out}}(D)=[b_{ij}']$ such that $b_{ij}'=1$ if $e_j=v_iv_k$ for some vertex $v_k$, and $b_{ij}'=0$ otherwise.
The next lemma shows how the product of the in- and out-incidence matrices of a digraph relates to the adjacency matrix of the digraph and its line digraph.

\begin{lemma}\label[lemma]{lem:bb}{\rm\cite{Brualdi,MR2312328}}
Let $D$ be a digraph. Then $A(D) = B_{\textnormal{out}}(D) B_{\textnormal{in}}(D)^T$. Furthermore, the adjacency matrix of its line digraph $\mathcal{L}(D)$ satisfies $A(\mathcal{L}(D)) = B_{\textnormal{in}}(D)^T B_{\textnormal{out}}(D)$ and $\lambda^n f_{A(\mathcal{L}(D))}(\lambda)=\lambda^{m} f_{A(D)}(\lambda)$.
\end{lemma}

We use $I_n$ to denote the $n\times n$ identity matrix, $O_{n_1,n_2}$ to denote the $n_1\times n_2$ zero matrix, 
$J_{n_1,n_2}$ to denote the $n_1\times n_2$ matrix with every entry equal to one,
along with the shorthand $J_n$ to denote $J_{n,n}$, $O_n$ to denote $O_{n,n}$ and $\textbf{1}_n$ to denote $J_{n,1}$.

McLeman and McNicholas \cite{Mcleman} introduced the coronal of a graph in order to describe the spectrum of the corona product of two graphs. 
This concept has also been extended to the coronal of a matrix \cite{cui2012spectrum}.
In particular, for a real $n\times n$ matrix $M$ considered as a matrix over the field of rational functions $\mathbb{C}(\lambda)$ with $\det(\lambda I_n-M)$ nonzero, the \emph{$M$-coronal} $\chi_M(\lambda) \in \mathbb{C}(\lambda)$ of $M$ is defined to be the sum of the entries of the matrix $(\lambda I_n-M)^{-1}$, i.e., $\chi_M(\lambda) = \textbf{1}_n^T\big(\lambda I_n-M\big)^{-1}\textbf{1}_n$. 
In the case where $M$ is the adjacency (resp. Laplacian or signless Laplacian) matrix of a digraph $D$, we refer to the $M$-coronal as the $A$-coronal (resp. $L$-coronal and $Q$-coronal) of $D$.

Tools from matrix theory commonly used to study operations on graphs and digraphs include Schur complements, a consequence of Sylvester's determinant identity and the Sherman-Morrison-Woodbury formula. These are summarized in the following three lemmas.

\begin{lemma}\label[lemma]{lem:schur}{\rm\cite[Eqn.~(6.2.1)]{meyer2023matrix}}
Let $ M_1, M_2, M_3 $, and $ M_4 $ be respectively $ p \times p $, $ p \times q $, $ q \times p $, and $ q \times q $ matrices.
If $M_4$ is invertible, then
\[\det \begin{bmatrix} 
    M_1 & M_2 \\
    M_3 & M_4 
\end{bmatrix} = \det(M_4) \cdot \det \left( M_1 - M_2 M_4^{-1} M_3 \right).\]
\end{lemma}
 
\begin{lemma}\label[lemma]{sylv}{\rm\cite[Eqn.~(6.2.3)]{meyer2023matrix}}
Let $C$ be an $n\times n$ invertible matrix and $\alpha\in\mathbb{R}$. Then 
\[
\det(C+\alpha J_n)=\det(C)\,(1+\alpha\,\textbf{1}_n^TC^{-1}\textbf{1}_n).
\]
\end{lemma}

\begin{lemma}\label[lemma]{woodbury}{\rm\cite[Eqn.~(3.8.3)]{meyer2023matrix}}
Let $A$ be an invertible $n\times n$ matrix, 
and let $U$ and $V$ be $n\times m$ and $m\times n$ matrices, respectively.
Then
\[
\left(A+UV\right)^{-1}=A^{-1}-A^{-1}U\left(I_m+VA^{-1}U\right)^{-1}VA^{-1}.
\]
In particular, when $A=\lambda I_n$, we have
$
\left(\lambda I_n+UV\right)^{-1}=\frac{1}{\lambda} I_n-\frac{1}{\lambda }U\left(\lambda I_m+VU\right)^{-1}V.
$
\end{lemma}

The following result gives a relationship between the coronal (resp.\ characteristic polynomial) of an $n\times n$ matrix $M$ and that of a linear combination of $M$, $I_n$, and $J_n$.

\begin{lemma}\label[lemma]{cor_addI}{\rm\cite[Theorem~2.2]{cavers2025spectra}}
Let $M$ be an $n\times n$ matrix and $M'=aM+bJ_n+cI_n$, where $a,b,c\in\mathbb{R}$ and $a\neq 0$.
Then the $M'$-coronal and the characteristic polynomial of $M'$ are, respectively, 
\begin{enumerate}
    \item $\displaystyle\chi_{M'}(\lambda)= \frac{\chi_M\left(\tfrac{\lambda-c}{a}\right)}{a-b\, \chi_M\left(\tfrac{\lambda-c}{a}\right)}$,
    \item $\displaystyle f_{M'}(\lambda)= a^{n-1} f_M\left(\tfrac{\lambda-c}{a}\right)\left(a-b\, \chi_M\left(\tfrac{\lambda-c}{a}\right)\right).$
\end{enumerate} 
\end{lemma}

When $M$ has constant row sum, a closed-form expression for the $M$-coronal is known.

\begin{lemma}\label[lemma]{coronal_cal}{\rm\cite[Proposition 2]{cui2012spectrum}}
Let $M$ be an $n\times n$ matrix such that each row sum of $M$ is equal to a constant $t$. Then $\chi_M (\lambda) = n/(\lambda-t)$.
\end{lemma}

\Cref{coronal_cal} has been extended to equitable partitions in \cite{cavers2025spectra}, where additional coronal computations for digraph complements, joins of digraphs and directed paths are also derived. 
We provide another extension of this lemma in the following proposition.

\begin{proposition}\label[proposition]{coronalBC}
Let $B$ be an $n\times m$ matrix such that each column sum of $B$ is equal to a constant $t_1$, and 
$C$ be an $m\times n$ matrix such that each row sum of $C$ is equal to a constant $t_2$.
Then
\[
\chi_{BC}(\lambda) 
= \frac{n}{\lambda} + \frac{t_1 t_2}{\lambda}\, \chi_{CB}(\lambda).
\]
\end{proposition}

\begin{proof}
Since $\mathbf{1}_n^T B = t_1\,\mathbf{1}_m^T$ and $C\,\mathbf{1}_n = t_2\,\mathbf{1}_m$, it follows by \cref{woodbury} (taking $U=-B$ and $V=C$) that
\[
\chi_{BC}(\lambda) 
= \mathbf{1}_n^T (\lambda I_n - BC)^{-1} \mathbf{1}_n
= \frac{1}{\lambda}\mathbf{1}_n^T\mathbf{1}_n + \frac{t_1 t_2}{\lambda}\, \mathbf{1}_m^T\left(\lambda I_m - CB\right)^{-1}\mathbf{1}_m
= \frac{n}{\lambda} + \frac{t_1 t_2}{\lambda}\, \chi_{CB}(\lambda),
\] 
as required.
\end{proof}

Note that \cref{coronal_cal} is a special case of \cref{coronalBC} by taking $C$ to an $n\times n$ matrix with constant row sum and $B=I_n$. 
Given a graph $G$, \cref{coronalBC} gives the following relationship between the $Q$-coronal of $G$ and the $A$-coronal of its line graph $\mathcal{L}(G)$.

\begin{corollary}\label[corollary]{linecorgraph}
Let $G$ be a graph on $n$ vertices with $m$ edges, $Q(G)$ its signless Laplacian matrix and $A(\mathcal{L}(G))$ the adjacency matrix of its line graph.
Then
\[
\chi_{Q(G)}(\lambda)=\frac{n}{\lambda}+\frac{4}{\lambda}\chi_{A(\mathcal{L}(G))}(\lambda-2).
\]
\end{corollary}

\begin{proof}
Suppose $B(G)$ is the $n\times m$ incidence matrix of $G$.
Note the signless Laplacian matrix satisfies $Q(G)=B(G)B(G)^T$ and the line graph $\mathcal{L}(G)$ satisfies $A(\mathcal{L}(G))=B(G)^TB(G)-2I_m$.
Since each column sum of $B(G)$ and each row sum of $B(G)^T$ is equal to $2$, by \cref{coronalBC}, 
we have
\[
\chi_{Q(G)}(\lambda)
=\frac{n}{\lambda}+\frac{4}{\lambda}\chi_{A(\mathcal{L}(G))+2I_m}(\lambda).
\]
The result now follows as $\chi_{A(\mathcal{L}(G))+2I_m}(\lambda)=\chi_{A(\mathcal{L}(G))}(\lambda-2)$ by \cref{cor_addI}.
\end{proof}

Similarly, given a digraph $D$, \cref{coronalBC} can be used to obtain the following relationship between the $A$-coronal of $D$ and of its line digraph $\mathcal{L}(D)$.

\begin{corollary}\label[corollary]{linecor}
Let $D$ be a digraph on $n$ vertices and $\mathcal{L}(D)$ its line digraph.    
Then 
\[
\chi_{A(D)}(\lambda)=\frac{n}{\lambda}+\frac{1}{\lambda}\chi_{A(\mathcal{L}(D))}(\lambda).
\]
\end{corollary}

\begin{proof}
This follows directly from \cref{lem:bb}, \cref{coronalBC} and the fact that every column of $B_{\textnormal{out}}(D)$ and every row of $B_{\textnormal{in}}(D)^T$ has exactly one $1$ and $0$'s otherwise.
\end{proof}

The following proposition and its corollary will be used in determining the spectrum of subdivision products of digraphs and follow directly from \cref{lem:schur} and properties of Kronecker products.

\begin{proposition}\label[proposition]{prop:M1-charpoly-ones}
Let $r,s,t\geq 1$, $\ell\in\{1,r\}$ and suppose
\[
M=
\begin{bmatrix}
A & B & \pm I_\ell\otimes J_{\frac{r}{\ell},t}\\
C & D & O_{s,\ell t}\\
\pm I_\ell\otimes J_{t,\frac{r}{\ell}} & O_{\ell t,s} & I_\ell\otimes E
\end{bmatrix},
\]
where $A$, $B$, $C$, $D$, and $E$ are 
$r\times r$, $r\times s$, $s\times r$, $s\times s$, and $t\times t$ real matrices, respectively.
Then 
\[
f_{M}(\lambda)
=f_D(\lambda) \; [f_E(\lambda)]^\ell \; 
\det\Bigl(
\lambda I_r - A - \chi_E(\lambda) \left(I_\ell\otimes J_{\frac{r}{\ell}}\right) - B(\lambda I_s - D)^{-1} C
\Bigr).
\]
\end{proposition}

\begin{proof}
By \cref{lem:schur}, we have
$\det(\lambda I-M)=\det(\lambda I_{\ell t} - I_\ell\otimes E)\det(S)$,
where
\[
S =
\begin{bmatrix}
\lambda I_r - A & -B \\
-C & \lambda I_s - D
\end{bmatrix}
-
\begin{bmatrix}
\pm I_\ell \otimes J_{\frac{r}{\ell},t} \\[1pt] O_{s,\ell t}
\end{bmatrix}
(\lambda I_{\ell t} - I_\ell \otimes E)^{-1}
\begin{bmatrix}
\pm I_\ell \otimes J_{t,\frac{r}{\ell}} & O_{\ell t,s}
\end{bmatrix}.
\]
For the first factor in $\det(\lambda I-M)$, observe that
\begin{equation}\label{eq:1}
\det(\lambda I_{\ell t} - I_\ell\otimes E) = [\det(\lambda I_t - E)]^\ell = [f_E(\lambda)]^\ell.
\end{equation}
We next simplify the expression for $S$ using properties of the Kronecker product, namely
\begin{align*}
(\pm I_\ell\otimes J_{\frac{r}{\ell},t}) 
(\lambda I_{\ell t} - I_\ell \otimes E)^{-1}
(\pm I_\ell \otimes J_{t,\frac{r}{\ell}})
&= (I_\ell \otimes J_{\frac{r}{\ell},t})
(I_\ell \otimes (\lambda I_t - E)^{-1})
(I_\ell \otimes J_{t,\frac{r}{\ell}}) \\[4pt]
&= I_\ell \otimes
\bigl(J_{\frac{r}{\ell},t}(\lambda I_t - E)^{-1}J_{t,\frac{r}{\ell}}\bigr) \\[4pt]
&= I_\ell \otimes
\bigl(\mathbf 1_{\frac{r}{\ell}}\mathbf 1_t^T
(\lambda I_t - E)^{-1}
\mathbf 1_t\mathbf 1_{\frac{r}{\ell}}^T\bigr) \\[4pt]
&= \chi_E(\lambda)\,(I_\ell \otimes J_{\frac{r}{\ell}}).
\end{align*}
Hence,
\[
S =
\begin{bmatrix}
\lambda I_r - A - \chi_E(\lambda)\left(I_\ell\otimes J_{\frac{r}{\ell}}\right) & -B \\
-C & \lambda I_s - D
\end{bmatrix},
\]
and by \cref{lem:schur}, 
\begin{equation}\label{eq:2}
\det(S) = \det(\lambda I_s - D) \;
\det\Bigl(\lambda I_r - A - \chi_E(\lambda)\left(I_\ell \otimes J_{\frac{r}{\ell}}\right) - B (\lambda I_s - D)^{-1} C \Bigr).
\end{equation}
Since $\det(\lambda I_s-D)=f_D(\lambda)$, multiplying \eqref{eq:1} and \eqref{eq:2} gives the required result.
\end{proof}

\begin{corollary}\label[corollary]{cor:3by3}
Let $r,s,t\geq 1$ and $\alpha,\beta,\gamma\in\mathbb{R}$.
Suppose
\[
M_1=\begin{bmatrix}
\alpha I_r & B & \pm J_{r,t}\\
C & \beta I_s & O_{s,t}\\
\pm J_{t,r} & O_{t,s} & \gamma I_t+F
\end{bmatrix},\qquad
M_2 = \begin{bmatrix}
\alpha I_r & B & O_{r,t} \\
C & \beta I_s & \pm J_{s,t} \\
O_{t,r} & \pm J_{t,s} & (\gamma I_t+F)
\end{bmatrix},
\]
\[
M_3=\begin{bmatrix}
\alpha I_r & B & \pm I_r\otimes J_{1,t}\\
C & \beta I_s & O_{s,rt}\\
\pm I_r\otimes J_{t,1} & O_{rt,s} & I_r\otimes (\gamma I_t+F)
\end{bmatrix}\quad\text{and}\quad
M_4=\begin{bmatrix}
\alpha I_r & B & O_{r,st} \\
C & \beta I_s & \pm I_s\otimes J_{1,t} \\
O_{st,r} & \pm I_s\otimes J_{t,1} & I_s\otimes (\gamma I_t+F)
\end{bmatrix},
\]
where $B$, $C$ and $F$ are $r\times s$, $s\times r$ and $t\times t$ real matrices, respectively.
Then 
\begin{enumerate}
    \item $\displaystyle f_{M_1}(\lambda)=(\lambda - \beta)^{s-r}\; f_F(\lambda-\gamma)\; f_{BC}\bigl((\lambda-\alpha)(\lambda-\beta)\bigr)\;\Bigl(1 - (\lambda-\beta)\,\chi_F(\lambda-\gamma)\, \chi_{BC}\bigl((\lambda-\alpha)(\lambda-\beta)\bigr) \Bigr)$,
    \item $\displaystyle f_{M_2}(\lambda)=(\lambda - \alpha)^{r-s}\; f_F(\lambda-\gamma)\; f_{CB}\bigl((\lambda-\alpha)(\lambda-\beta)\bigr)\;\Bigl(1 - (\lambda-\alpha)\,\chi_F(\lambda-\gamma)\, \chi_{CB}\bigl((\lambda-\alpha)(\lambda-\beta)\bigr) \Bigr)$,
    \item $\displaystyle f_{M_3}(\lambda)= (\lambda - \beta)^{s-r}\; [f_F(\lambda-\gamma)]^r\; f_{BC}\Bigl((\lambda-\beta)\bigl(\lambda-\alpha- \chi_F(\lambda-\gamma)\bigr)\Bigr)$,    
    \item $\displaystyle f_{M_4}(\lambda)= (\lambda - \alpha)^{r-s}\; [f_F(\lambda-\gamma)]^s\; 
f_{CB}\Bigl((\lambda-\alpha)\bigl(\lambda-\beta- \chi_F(\lambda-\gamma)\bigr)\Bigr)$.
\end{enumerate}
\end{corollary}

\begin{proof}
Let $A=\alpha I_r$, $D=\beta I_s$ and $E=\gamma I_t+F$ in \cref{prop:M1-charpoly-ones}, and note that $f_{D}(\lambda)=(\lambda-\beta)^s$, $f_{E}(\lambda)=f_F(\lambda-\gamma)$
and $\chi_{E}(\lambda)=\chi_F(\lambda-\gamma)$
(the last two equations follow by \cref{cor_addI}).
The formula for $f_{M_1}(\lambda)$ follows by letting $\ell = 1$ since then $I_\ell\otimes J_{\frac{r}{\ell}} = J_r$, and by \cref{sylv}, we obtain
\begin{align*}
\det\Bigl(&\lambda I_r - \alpha I_r - \chi_F(\lambda-\gamma) J_r 
- B (\lambda I_s - \beta I_s)^{-1} C \Bigr) \\[1mm]
&= (\lambda - \beta)^{-r} \, 
\det\Bigl((\lambda-\alpha)(\lambda-\beta) I_r - BC 
- (\lambda - \beta)\, \chi_F(\lambda-\gamma) J_r \Bigr) \\[1mm]
&= (\lambda - \beta)^{-r} \, 
f_{BC}\bigl((\lambda-\alpha)(\lambda-\beta)\bigr)\,\Bigl(1 - (\lambda-\beta)\,\chi_F(\lambda-\gamma) \, 
\chi_{BC}\bigl((\lambda-\alpha)(\lambda-\beta)\bigr) \Bigr).
\end{align*} 
The formula for $f_{M_3}(\lambda)$ follows by letting $\ell = r$ since then $I_\ell\otimes J_{\frac{r}{\ell}} = I_r$, and we obtain
\begin{align*}
\det\Bigl(&\lambda I_r - \alpha I_r - \chi_F(\lambda-\gamma) I_r - B (\lambda I_s - \beta I_s)^{-1} C \Bigr) \\[1mm]
&=(\lambda - \beta)^{-r} \, \det\Bigl((\lambda-\beta)\bigl(\lambda-\alpha- \chi_F(\lambda-\gamma)\bigr) I_r - BC \Bigr) \\[1mm]
&=(\lambda - \beta)^{-r} \, f_{BC}\Bigl((\lambda-\beta)\bigl(\lambda-\alpha- \chi_F(\lambda-\gamma)\bigr)\Bigr).
\end{align*}
For $f_{M_2}(\lambda)$ (resp. $f_{M_4}(\lambda)$), let $P=\begin{bmatrix}
O & I & O\\
I & O & O \\
O & O & I
\end{bmatrix}$ and observe that $P^TM_1P$ (resp. $P^TM_3P$) has the same structure as $M_2$ (resp. $M_4$) except with $\alpha I_r$ and $\beta I_s$ swapped, $B$ and $C$ swapped, and $r$ and $s$ swapped. 
Since $f_{P^TMP}(\lambda)=f_{M}(\lambda)$ for any square matrix $M$, the result follows by swapping the appropriate labels previously mentioned in the formula for $f_{M_1}(\lambda)$ (resp. $f_{M_3}(\lambda)$). 
\end{proof}

The formulas remain valid with the evident interpretation when one of the zero-dimensional blocks is absent.
 
Finally, we note the following relationship between the characteristic polynomials (resp. coronals) of the adjacency, Laplacian and signless-Laplacian matrices of an out-regular digraph along with the adjacency matrix of its line digraph.

\begin{proposition}\label[proposition]{prop:regular}
If $D$ is an $r$-out-regular digraph on $n$ vertices and $r>0$, then 
\[
\chi_{A(D)}(\lambda)
=\chi_{L(D)}(\lambda-r)
=\chi_{Q(D)}(\lambda+r)
=\frac{1}{r}\,\chi_{A(\mathcal{L}(D))}(\lambda)=\frac{n}{\lambda-r}
\]
and
\[
f_{A(D)}(\lambda) = (-1)^n f_{L(D)}(r-\lambda) = f_{Q(D)}(\lambda+r) = \lambda^{n(1-r)}f_{A(\mathcal{L}(D))}(\lambda).
\]
\end{proposition}

\begin{proof}
If $D$ is $r$-out-regular, then $D_{\textnormal{out}}(D)=rI_n$, and hence, $L(D)=rI_n-A(D)$ and $Q(D)=rI_n+A(D)$.
The result for the coronal now follows by \cref{coronal_cal} and \cref{linecor}, while the result for the characteristic polynomial follows by \cref{lem:bb} and \cref{cor_addI}.
\end{proof}

\section{Subdivision graphs and digraphs}\label[section]{sec:subdiv}
For a graph $G$ with $n$ vertices and $m$ edges, the \emph{subdivision graph} $\mathcal{S}(G)$ of $G$ is obtained by replacing each edge of $G$ by a path of length two (e.g., see \cite{cvetkovic1980spectra}).
We denote the set of new vertices by $I(G)$. 
Up to permutation similarity, the adjacency matrix of $\mathcal{S}(G)$ is 
\[
A(\mathcal{S}(G))=\begin{bmatrix} 
    O_n & B(G) \\
    B(G)^T & O_m
\end{bmatrix},
\]
where $B(G)$ is the $n\times m$ incidence matrix of $G$.
Since the signless Laplacian matrix satisfies $Q(G)=B(G)B(G)^T$ and the line graph $\mathcal{L}(G)$ satisfies $A(\mathcal{L}(G))=B(G)^TB(G)-2I_m$ (e.g., see \cite{brouwer2011spectra}), we obtain the following known relationship (e.g., see \cite{shinoda1980chara} and also  \cite{cvetkovic1980spectra} for the case when $G$ is a regular graph) between the 
$A$-spectrum of $\mathcal{S}(G)$, the $Q$-spectrum of $G$ and the $A$-spectrum of the line graph $\mathcal{L}(G)$: 
\begin{equation}\label{eq:QG}
f_{A(\mathcal{S}(G))}(\lambda)=\lambda^{m-n}f_{Q(G)}(\lambda^2)
=\lambda^{n-m}f_{A(\mathcal{L}(G))}(\lambda^2-2).
\end{equation}
In the case where $G$ is a regular graph, adjacency, Laplacian and signless Laplacian eigenvalues along with eigenvector information for $\mathcal{S}(G)$ can be found in \cite{BarikKalita}.

In contrast, let $D$ be a digraph with $n$ vertices and $m$ arcs.
The \emph{subdivision digraph} $\mathcal{S}(D)$ of $D$ is obtained by replacing each arc of $D$ by a directed path of length two (see \cite{Brualdi} and also \cite{DengKelmans2013} where the notation $D^{00+}$ is used).
Note that $\mathcal{S}(D)$ is strongly connected if and only if $D$ is strongly connected.
We denote the set of new vertices by $I(D)$; thus $(V(D),I(D))$ is a bipartition for $\mathcal{S}(D)$.
Up to permutation similarity, the adjacency matrix of $\mathcal{S}(D)$ is 
\[
A(\mathcal{S}(D))=\begin{bmatrix} 
    O_n & B_{\textnormal{out}}(D) \\
    B_{\textnormal{in}}(D)^T & O_m
\end{bmatrix}.
\]
Since $A(D)=B_{\textnormal{out}}(D)B_{\textnormal{in}}(D)^T$ (by \cref{lem:bb}), we obtain the following known (e.g., see \cite{DengKelmans2013,zhang1987characteristic}) relationship between the $A$-spectra of the digraphs $\mathcal{S}(D)$, $D$ and $\mathcal{L}(D)$: 
\[
f_{A(\mathcal{S}(D))}(\lambda)
=\lambda^{\,m-n}\,f_{A(D)}(\lambda^2)
=\lambda^{\,n-m}\,f_{A(\mathcal{L}(D))}(\lambda^2).
\]
In particular, the adjacency spectrum of $\mathcal{S}(D)$ is completely determined by the adjacency spectrum of $D$.

We emphasize that subdivision in the directed setting differs from the undirected case in the following sense: 
If $D$ is a symmetric digraph (i.e., $A(D)$ is a symmetric matrix) with underlying graph $G=U(D)$, then each symmetric arc of $D$ (that is, each $2$-cycle) contributes two new vertices to $\mathcal{S}(D)$, one for each arc, whereas each edge of $G$ contributes a single new vertex to $\mathcal{S}(G)$.

\section{Subdivision-join products}\label[section]{sec:sub-join}
For graphs $G_1$ and $G_2$, the \emph{subdivision-vertex join} $G_1\svj G_2$ is the graph obtained from $\mathcal{S}(G_1)$ by joining every original vertex of $G_1$ to every vertex of $G_2$, and the \emph{subdivision-edge join} $G_1\saj G_2$ is obtained from $\mathcal{S}(G_1)$ by joining every newly inserted vertex of $\mathcal S(G_1)$ to every vertex of $G_2$.
These two graph operations were introduced by Indulal \cite{Indulal} who computed their adjacency spectrum when both $G_1$ and $G_2$ are regular.
Using coronals, Liu and Zhang~\cite{Liu} extended this work by allowing $G_2$ to be arbitrary while still assuming that $G_1$ is regular, and they computed the adjacency, Laplacian, and signless Laplacian spectra of $G_1\svj G_2$ and $G_1\saj G_2$.
Later, Pavithra and Rajkumar~\cite{Pavithra2021} considered the case where $G_1$ is the join of two graphs with $G_2$ arbitrary in the subdivision-vertex join $G_1\svj G_2$.

Before extending these operations to the directed setting, we first establish a result for graphs that describes the $A$-spectrum of $G_1\svj G_2$ (resp. $G_1\saj G_2$) when both factors $G_1$ and $G_2$ are arbitrary. 

\begin{theorem}
Let $G_1$ and $G_2$ be simple graphs on $n_1$ and $n_2$ vertices, and $m_1$ and $m_2$ edges, respectively. 
Then
\[
f_{A(G_1\svj G_2)}(\lambda)
= \lambda^{n_1-m_1-1}
\,f_{A(\mathcal{L}(G_1))}(\lambda^2-2)
\,f_{A(G_2)}(\lambda)
\Bigl[\lambda-\Bigl(n_1+4\,\chi_{A(\mathcal{L}(G_1))}(\lambda^2-2)\Bigr)\chi_{A(G_2)}(\lambda)\Bigr],
\]
and
\[
f_{A(G_1\saj G_2)}(\lambda)
= \lambda^{n_1 - m_1}\,
f_{A(\mathcal{L}(G_1))}(\lambda^2-2)\,
f_{A(G_2)}(\lambda)\, \Bigl( 1 - \lambda \, \chi_{A(\mathcal{L}(G_1))}(\lambda^2-2)\,
\chi_{A(G_2)}(\lambda) 
 \Bigr),
\]
where $\mathcal{L}(G_1)$ is the line graph of $G_1$.
In particular, 
the $A$-spectrum of $G_1\saj G_2$ (resp. $G_1\svj G_2$) is completely determined by the $A$-spectra of $\mathcal{L}(G_1)$, and $G_2$ and the $A$-coronals of $\mathcal{L}(G_1)$ and $G_2$.
\end{theorem} 

\begin{proof}
Up to permutation similarity and with respect to the vertex partition $(V(G_1), I(G_1), V(G_2))$, the adjacency matrices of
$G_1\svj G_2$ and $G_1\saj G_2$ are
\[
A(G_1\svj G_2)=
\begin{bmatrix}
O_{n_1} & B(G_1) & J_{n_1,n_2} \\
B(G_1)^T & O_{m_1} & O_{m_1,n_2} \\
J_{n_2,n_1} & O_{n_2,m_1} & A(G_2)
\end{bmatrix},\quad\text{and}\quad
A(G_1\saj G_2)=
\begin{bmatrix}
O_{n_1} & B(G_1) &  O_{n_1,n_2}\\
B(G_1)^T & O_{m_1} & J_{m_1,n_2} \\
O_{n_2,n_1} & J_{n_2,m_1} & A(G_2)
\end{bmatrix},
\]
respectively, where $B(G_1)$ is the incidence matrix of $G_1$.
The result now follows by \cref{cor:3by3}(1) (resp. \cref{cor:3by3}(2)) and using $Q(G_1)=B(G_1)B(G_1)^T$ (resp. $A(\mathcal{L}(G_1))=B(G_1)^TB(G_1)-2I_{m_1}$), along with \cref{cor_addI}, \eqref{eq:QG} and \cref{linecorgraph}.
\end{proof}

\subsection{The subdivision-vertex join of two digraphs} 

Let $D_1$ and $D_2$ be digraphs.  
The \emph{(symmetric) subdivision-vertex join} $D_1 \svj D_2$ is obtained by taking $\mathcal{S}(D_1)$ and $D_2$, and adding both arcs $v w$ and $w v$ for every original vertex $v \in V(D_1)$ and every vertex $w \in V(D_2)$ (see \cref{fig:P2subvjoin}(b)).  
One can similarly define the forward (only arcs $v w$) or backward (only arcs $w v$) subdivision-vertex joins, however, these are not strongly connected and are therefore not considered here.

\begin{figure}[ht]
\centering
\hspace*{-0.5cm}
\begin{subfigure}{0.32\textwidth}
\centering
\scalebox{0.95}{
\raisebox{2.2cm}{
\begin{tikzpicture}[baseline=(v2.base), >=Stealth, thick, scale=1, vertex/.style={circle, fill, inner sep=2pt}, 
ivertex/.style={circle, draw, inner sep=2pt},
midarrow/.style={line width=1.3pt, postaction={decorate, decoration={markings, mark=at position 0.599 with {\arrow{Stealth[length=3mm,width=3mm]}}}}}, 
doublearrow/.style={line width=2.0pt, red, dash pattern=on 1pt off 1pt}]
\usetikzlibrary{decorations.markings,arrows.meta}
\draw[blue, dashed] (-0.4,-0.4) rectangle (3.4,0.4)
                    (0.3,1.6) rectangle (2.7,2.4);
\node[left] at (-0.5,0) {$V(D_1)$};
\node[left] at (0.2,2) {$I(D_1)$};
\node[vertex] (v1) at (0,0) {};
\node[vertex] (v2) at (1.5,0) {};
\node[vertex] (v3) at (3,0) {};
\node[ivertex] (i1) at (0.75,2) {};
\node[ivertex] (i2) at (2.25,2) {};
\foreach \a/\b in {v1/i1, i1/v2, v2/i2, i2/v3}
  \draw[midarrow] (\a) -- (\b);
\end{tikzpicture}}}
\caption{The subdivision digraph $\mathcal{S}(D_1)$}
\end{subfigure}
\hfill
\begin{subfigure}{0.32\textwidth}
\centering
\scalebox{0.95}{
\raisebox{1.0cm}{
\begin{tikzpicture}[baseline=(v2.base), >=Stealth, thick, scale=1, vertex/.style={circle, fill, inner sep=2pt}, 
ivertex/.style={circle, draw, inner sep=2pt},
midarrow/.style={line width=1.3pt, postaction={decorate, decoration={markings, mark=at position 0.59 with {\arrow{Stealth[length=3mm,width=3mm]}}}}}, 
doublearrow/.style={line width=2.0pt, red, dash pattern=on 1pt off 1pt}]
\usetikzlibrary{decorations.markings,arrows.meta}
\draw[blue, dashed] (-0.4,-0.4) rectangle (3.4,0.4)
                    (0.3,1.6) rectangle (2.7,2.4);
\node[left] at (-0.5,0) {$V(D_1)$};
\node[left] at (0.2,2) {$I(D_1)$};
\node[vertex] (v1) at (0,0) {};
\node[vertex] (v2) at (1.5,0) {};
\node[vertex] (v3) at (3,0) {};
\node[ivertex] (i1) at (0.75,2) {};
\node[ivertex] (i2) at (2.25,2) {};
\coordinate (d2) at (1.5,-1.7);
\foreach \a/\b in {v1/i1, i1/v2, v2/i2, i2/v3}
  \draw[midarrow] (\a) -- (\b);
\coordinate (e1) at ($(d2)+(-1.2+0.15,0.26)$);
\coordinate (e2) at ($(d2)+(1.2-0.15,0.26)$);
\foreach \a in {v1,v2,v3}
  \draw[doublearrow] (\a) -- (e1)
                     (\a) -- (e2);
\node[draw, ellipse, line width=1.3pt, minimum width=2.4cm, minimum height=1cm] at (d2) {$D_2$};
\end{tikzpicture}}}
\caption{The subdivision-vertex join $D_1\svj D_2$}
\end{subfigure}
\hfill
\begin{subfigure}{0.32\textwidth}
\centering
\scalebox{0.95}{
\raisebox{2.2cm}{
\begin{tikzpicture}[baseline=(v2.base), >=Stealth, thick, scale=1, vertex/.style={circle, fill, inner sep=2pt}, 
ivertex/.style={circle, draw, inner sep=2pt},
midarrow/.style={line width=1.3pt, postaction={decorate, decoration={markings, mark=at position 0.59 with {\arrow{Stealth[length=3mm,width=3mm]}}}}}, 
doublearrow/.style={line width=2.0pt, red, dash pattern=on 1pt off 1pt}]
\usetikzlibrary{decorations.markings,arrows.meta}
\draw[blue, dashed] (-0.4,-0.4) rectangle (3.4,0.4)
                    (0.3,1.6) rectangle (2.7,2.4);
\node[left] at (-0.5,0) {$V(D_1)$};
\node[left] at (0.2,2) {$I(D_1)$};
\node[vertex] (v1) at (0,0) {};
\node[vertex] (v2) at (1.5,0) {};
\node[vertex] (v3) at (3,0) {};
\node[ivertex] (i1) at (0.75,2) {};
\node[ivertex] (i2) at (2.25,2) {};
\coordinate (d2) at (1.5,3.4);
\foreach \a/\b in {v1/i1, i1/v2, v2/i2, i2/v3}
  \draw[midarrow] (\a) -- (\b);
\coordinate (e1) at ($(d2)+(-1.2+0.15,-0.26)$);
\coordinate (e2) at ($(d2)+( 1.2-0.15,-0.26)$);
\foreach \a in {i1,i2}
  \draw[doublearrow] (\a) -- (e1)
                     (\a) -- (e2);
\node[draw, ellipse, line width=1.3pt, minimum width=2.4cm, minimum height=1cm] at (d2) {$D_2$};
\end{tikzpicture}}}
\caption{The subdivision-arc join $D_1\saj D_2$}
\end{subfigure}
\caption{Above, $D_1=P_3$ (the directed path on three vertices) and $D_2$ is an arbitrary digraph. The dashed (red) lines represent all possible arcs in both directions between vertices in $V(D_2)$ and vertices in $V(D_1)$ (for $D_1\svj D_2$) or $I(D_1)$ (for $D_1\saj D_2$).}
\label{fig:P2subvjoin}
\end{figure}
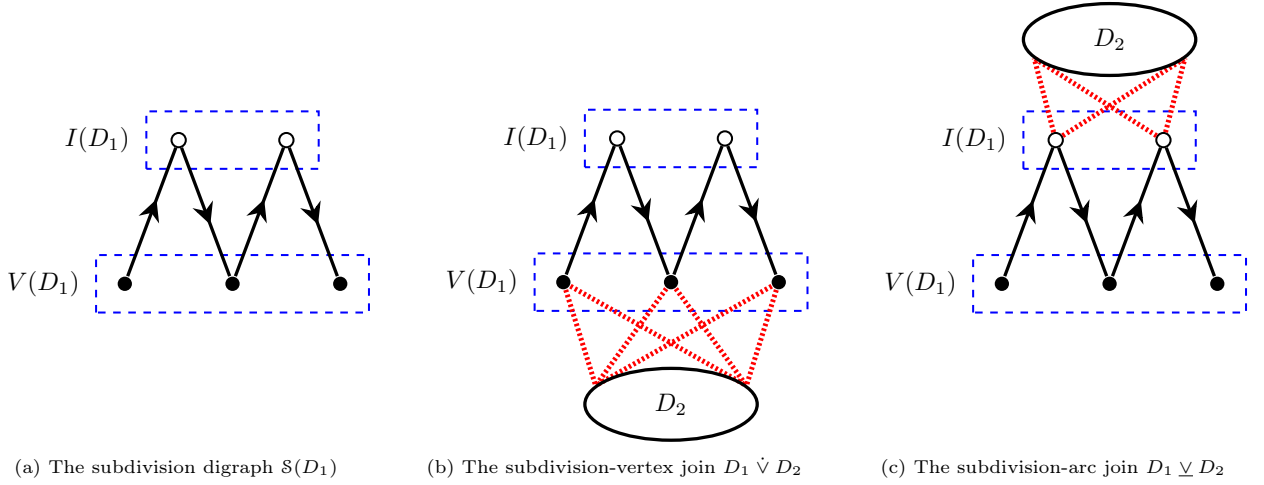
 
In the graph case, the subdivision-vertex join $G_1\svj G_2$ is always connected, and an analogous statement holds for digraphs.

\begin{proposition}
Let $D_1$ and $D_2$ be digraphs.
Then $D_1\svj D_2$ is strongly connected.    
\end{proposition}

\begin{proof}
Let $D=D_1\svj D_2$.
Observe that $V(D)=I(D_1)\cup V(D_1)\cup V(D_2)$ and every vertex in $V(D_1)$ is connected to every vertex in $V(D_2)$ by a $2$-cycle in $D$, i.e., $D$ has a symmetric complete bipartite digraph as a subdigraph with bipartition $(V(D_1),V(D_2))$. 
Thus, for any $u,v\in V(D_1)\cup V(D_2)$, there is a directed path from $u$ to $v$ in $D$. 
Since every vertex $w\in I(D_1)$ is the head (resp. tail) of an arc whose other endpoint is in $V(D_1)$, every vertex of $I(D_1)$ both reaches and is reachable from $V(D_1)$, and hence, both reaches and is reachable from every vertex of $D$. 
Therefore, $D$ is strongly connected.
\end{proof}

\begin{theorem}\label[theorem]{thm:sub-vertex}
Let $D_1$ and $D_2$ be two digraphs on $n_1$ and $n_2$ vertices, and $m_1$ and $m_2$ arcs, respectively. Then
$$\displaystyle f_{A(D_1\svj D_2)}(\lambda)=\lambda^{m_1-n_1}\, f_{A(D_1)}(\lambda^2)\, f_{A(D_2)}(\lambda)\,
    \left[1-\lambda\;\chi_{A(D_1)}(\lambda^2)\;\chi_{A(D_2)}(\lambda)\right].$$
In particular, the $A$-spectrum of $D_1\svj D_2$ is completely determined by the $A$-spectra of $D_1$ and $D_2$, and the $A$-coronals of $D_1$ and $D_2$.   
\end{theorem}

\begin{proof}
   The adjacency matrix for $D_1\svj D_2$ can be written in the block form
\[
A(D_1\svj D_2) = \begin{bmatrix}
    O_{n_1} & B_{\text{out}}(D_1) & J_{n_1, n_2} \\
    B_{\text{in}}(D_1)^T & O_{m_1} & O_{m_1, n_2} \\
    J_{n_2, n_1} & O_{n_2, m_1} & A(D_2)
\end{bmatrix}.
\]
The result now follows by \cref{cor:3by3}(1) 
and noting $A(D_1)=B_{\text{out}}(D_1)B_{\text{in}}(D_1)^T$
by \cref{lem:bb}.
\end{proof}

\begin{corollary}\label[corollary]{cor:regular}
Let $D_1$ be an $r$-out-regular digraph on $n_1$ vertices and $D_2$ be an $s$-out-regular digraph on $n_2$ vertices. Then
\[
f_{A(D_1\svj D_2)}(\lambda) = \lambda^{n_1(r-1)} f_{A(D_1)}(\lambda^2) f_{A(D_2)}(\lambda) \left[1 - \frac{\lambda n_1 n_2}{(\lambda^2 - r)(\lambda - s)}\right].
\]
\end{corollary}

\begin{proof}
Since $D_1$ is $r$-out-regular, $D_1$ has $m_1 = n_1 r$ arcs and $\chi_{A(D_1)}(\lambda^2) = \frac{n_1}{\lambda^2 - r}$. Similarly, since $D_2$ is $s$-out-regular, $\chi_{A(D_2)}(\lambda) = \frac{n_2}{\lambda - s}$. Substituting these into \cref{thm:sub-vertex} gives
the result. 
\end{proof}

In the case that $D_1$ is out-regular, we can obtain simple formulas for the Laplacian and signless Laplacian characteristic polynomial for $D_1\svj D_2$. 

\begin{theorem}\label[theorem]{laplacian-sub-vjoin}
Let $D_1$ and $D_2$ be two digraphs on $n_1$ and $n_2$ vertices, and $m_1$ and $m_2$ arcs, respectively.
Suppose that $D_1$ is $r$-out-regular.
Then 
\begin{enumerate}
    \item $\displaystyle f_{L(D_1\svj D_2)}(\lambda)
=\\
(-1)^{n_1}(\lambda-1)^{n_1(r-1)}\,
f_{L(D_2)}(\lambda-n_1)\,
f_{L(D_1)}\bigl(r-(\lambda-r-n_2)(\lambda-1)\bigr)
\left[
  1-\frac{(\lambda-1)n_1n_2}{(\lambda-n_1)((\lambda-n_2-r)(\lambda-1)-r)}
\right]$, 
    \item $\displaystyle f_{Q(D_1\svj  D_2)
}(\lambda)=(\lambda-1)^{n_1(r-1)}\,
f_{Q(D_2)}(\lambda-n_1)\,
f_{Q(D_1)}\bigl(r+(\lambda-r-n_2)(\lambda-1)\bigr)
\left[
  1-\frac{(\lambda-1)n_1\,\chi_{Q(D_2)}(\lambda-n_1)}{(\lambda-r-n_2)(\lambda-1)-r}
\right].$
\end{enumerate}    
In particular, the $L$-spectrum of $D_1\svj D_2$ is completely determined by the $L$-spectra of $D_1$ and $D_2$,
and the $Q$-spectrum of $D_1\svj D_2$ is completely determined by the $Q$-spectra of $D_1$ and $D_2$, and the $Q$-coronal of $D_2$.
\end{theorem}
 
\begin{proof}
Since $D_1$ is $r$-out-regular, it has $m_1=n_1r$ arcs and $D_{\textnormal{out}}(D_1)=rI_{n_1}$.
Thus, the Laplacian matrix of $D_1\svj D_2$ can be written in the block form
\[
L(D_1\svj D_2)=
\begin{bmatrix}
(n_2+r)I_{n_1} & -B_{\textnormal{out}}(D_1) & -J_{n_1,n_2}\\
-B_{\textnormal{in}}(D_1)^{T} & I_{m_1} & O_{m_1,n_2}\\
-J_{n_2,n_1} & O_{n_2,m_1} & n_1I_{n_2}+L(D_2)
\end{bmatrix}.
\]
The result now follows by \cref{cor:3by3}(1) and noting $A(D_1)=B_{\text{out}}(D_1)B_{\text{in}}(D_1)^T$
by \cref{lem:bb} and $\chi_{L(D_2)}(\lambda)=n_2/\lambda$
by \cref{coronal_cal}. 
Finally, $L(D_1)=rI_{n_1}-A(D_1)$ implies $f_{A(D_1)}(\lambda)=(-1)^{n_1}f_{L(D_1)}(r-\lambda)$. 

For the second equation, the signless Laplacian matrix of $D_1\svj D_2$ can be written in the block form
\[
Q(D_1\svj D_2)=
\begin{bmatrix}
(n_2+r)I_{n_1} & B_{\textnormal{out}}(D_1) & J_{n_1,n_2}\\
B_{\textnormal{in}}(D_1)^{T} & I_{m_1} & O_{m_1,n_2}\\
J_{n_2,n_1} & O_{n_2,m_1} & n_1I_{n_2}+Q(D_2)
\end{bmatrix}.
\]
The result now follows by \cref{cor:3by3}(1) and noting $A(D_1)=B_{\text{out}}(D_1)B_{\text{in}}(D_1)^T$
by \cref{lem:bb}. 
Finally, $Q(D_1)=rI_{n_1}+A(D_1)$ implies $f_{A(D_1)}(\lambda)=f_{Q(D_1)}(\lambda+r)$.  
\end{proof}

\subsection{The subdivision-arc join of two digraphs}

Let $D_1$ and $D_2$ be digraphs.  
The \emph{(symmetric) subdivision-arc join} $D_1 \saj D_2$ is obtained by taking $\mathcal{S}(D_1)$ and $D_2$, and adding both arcs $v w$ and $w v$ for every new vertex $v \in I(D_1)$ and every vertex $w \in V(D_2)$  (see \cref{fig:P2subvjoin}(c)).  
One can similarly define the forward (only arcs $v w$) or backward (only arcs $w v$) subdivision-arc joins, however, these are not strongly connected and are therefore not considered here.

For graphs $G_1$ and $G_2$, the subdivision-edge join $G_1\saj G_2$ is connected if and only if $G_1$ has no isolated vertices, i.e., $\delta(G_1)\geq 1$.
An analogous result holds for digraphs.
For a digraph $D$, we let $\delta^+(D)$ (resp. $\delta^-(D)$) denote the minimum out-degree (resp. in-degree) of $D$.

\begin{proposition}
Let $D_1$ and $D_2$ be digraphs.
Then $D_1\saj D_2$ is strongly connected if and only if 
$\delta^+(D_1)\geq 1$
and $\delta^-(D_1)\geq 1$.
\end{proposition}

\begin{proof}
Let $D=D_1\saj D_2$. If there exists $v\in V(D_1)$ with
$d^+_{D_1}(v)=0$, then $d^+_D(v)=0$, since vertices of $V(D_1)$
are not joined directly to vertices of $D_2$. Similarly, if
$d^-_{D_1}(v)=0$, then $d^-_D(v)=0$. Hence $D$ cannot be strongly
connected unless $\delta^+(D_1)\ge 1$ and $\delta^-(D_1)\ge 1$.

Conversely, suppose that $\delta^+(D_1)\ge 1$ and
$\delta^-(D_1)\ge 1$. Let $H=I(D_1)\cup V(D_2)$.
Every vertex of $I(D_1)$ is joined by a $2$-cycle to every vertex of $V(D_2)$. Since $D_2$ is nonempty and $\delta^+(D_1)\ge 1$, the set
$H$ is nonempty and induces a strongly connected subdigraph of $D$.

Now let $x\in V(D_1)$. Since $d^+_{D_1}(x)\ge 1$, there is an arc
$xy\in E(D_1)$. Thus, in $\mathcal S(D_1)$, $x$ reaches the
subdivision vertex corresponding to $xy$, which belongs to $H$. Hence
$x$ reaches $H$. Similarly, since $d^-_{D_1}(x)\ge 1$, there is an
arc $zx\in E(D_1)$. The subdivision vertex corresponding to $zx$
belongs to $H$ and reaches $x$. Hence $H$ reaches $x$.

Therefore every vertex of $V(D_1)$ both reaches and is reachable from the
strongly connected subdigraph induced by $H$. It follows that $D$ is
strongly connected.
\end{proof}

\begin{theorem}\label{thm:subdiv_A}
Let $D_1$ and $D_2$ be two digraphs on $n_1$ and $n_2$ vertices, and $m_1$ and $m_2$ arcs, respectively. 
Then 
\[
f_{A(D_1\saj D_2)}(\lambda)
=\lambda^{m_1-n_1}\, 
f_{A(D_1)}(\lambda^2)\, 
f_{A(D_2)}(\lambda)\,\left[1-\lambda\,
\left(\lambda^2\,\chi_{A(D_1)}(\lambda^2)-n_1\right)\,
\chi_{A(D_2)}(\lambda)\right].
\]
In particular, the $A$-spectrum of $D_1\saj D_2$ is completely determined by the $A$-spectra of $D_1$ and $D_2$, and the $A$-coronals of $D_1$ and $D_2$.
\end{theorem}

\begin{proof}
The adjacency matrix of $D_1\saj D_2$ can be written in the block form
\[
A(D_1\saj D_2)=\left[\begin{array}{ccc}
O_{n_1} & B_{\textnormal{out}}(D_1) & O_{n_1,n_2} \\
B_{\textnormal{in}}(D_1)^T & O_{m_1} &  J_{m_1,n_2} \\
O_{n_2,n_1} & J_{n_2,m_1} & A(D_2)
\end{array}\right].
\]
The result now follows by 
\cref{cor:3by3}(2), along with $A(\mathcal{L}(D_1))=B_{\text{in}}(D_1)^TB_{\text{out}}(D_1)$,
and $f_{A(\mathcal{L}(D_1))}(\lambda^2)=\lambda^{2m_1-2n_1}f_{A(D_1)}(\lambda^2)$
by \cref{lem:bb}, and \cref{linecor}.
\end{proof}

\begin{corollary}\label[corollary]{cor:regular_subdiv}
Let $D_1$ be an $r$-out-regular digraph on $n_1$ vertices and $D_2$ be an $s$-out-regular digraph on $n_2$ vertices. 
Then 
\[
f_{A(D_1\saj D_2)}(\lambda) 
= \lambda^{n_1(r-1)}\,
f_{A(D_1)}(\lambda^2)\,
f_{A(D_2)}(\lambda)
\left[1 - \frac{\lambda\, r\, n_1 n_2}{(\lambda^2 - r)(\lambda - s)}\right].
\]
\end{corollary}

\begin{proof}
Since $D_1$ is $r$-out-regular, it has $m_1 = n_1 r$ arcs, and furthermore, its line digraph $\mathcal{L}(D_1)$ is also $r$-out-regular.
Thus, $\chi_{A(\mathcal{L}(D_1))}(\lambda^2) = \frac{m_1}{\lambda^2 - r} = \frac{n_1 r}{\lambda^2 - r}$. Since $D_2$ is $s$-out-regular, $\chi_{A(D_2)}(\lambda) = \frac{n_2}{\lambda - s}$. Substituting these into Theorem \ref{thm:subdiv_A} yields the result.
\end{proof}

In the case that $D_1$ is out-regular, we can obtain simple formulas for the Laplacian and signless Laplacian characteristic polynomial for $D_1\saj D_2$.

\begin{theorem}
Let $D_1$ and $D_2$ be two digraphs on $n_1$ and $n_2$ vertices, and $m_1$ and $m_2$ arcs, respectively.
Suppose that $D_1$ is $r$-out-regular.
Then 
\begin{enumerate}
    \item $ f_{L(D_1\saj D_2)}=$\\
$
(-1)^{n_1}
(\lambda-1-n_2)^{n_1(r-1)}\,
f_{L(D_2)}(\lambda-n_1r)\,
f_{L(D_1)}\!\left(r-(\lambda-r)(\lambda-1-n_2)\right)\,
\left[
1-\frac{r(\lambda-r)n_1n_2}
{(\lambda-n_1r)\left((\lambda-r)(\lambda-1-n_2)-r\right)}
\right].
$
    \item $\displaystyle f_{Q(D_1\saj D_2)}(\lambda)=$\\
$\displaystyle 
(\lambda-1-n_2)^{n_1(r-1)}
f_{Q(D_2)}(\lambda-n_1r)\,
f_{Q(D_1)}\!\left(r+(\lambda-r)(\lambda-1-n_2)\right)\,
\left[
1-\frac{n_1r(\lambda-r)\,\chi_{Q(D_2)}(\lambda-n_1r)}
{(\lambda-r)(\lambda-1-n_2)-r}
\right].
$
\end{enumerate}    
In particular, the $L$-spectrum of $D_1\saj D_2$ is completely determined by the $L$-spectra of $D_1$ and $D_2$,
and the $Q$-spectrum of $D_1\saj D_2$ is completely determined by the $Q$-spectra of $D_1$ and $D_2$, and the $Q$-coronal of $D_2$.
\end{theorem}

\begin{proof}
Since $D_1$ is $r$-out-regular, it has $m_1 = n_1 r$ arcs and $D_{\textnormal{out}}(D_1)=rI_{n_1}$.
Thus, the Laplacian matrix of $D_1\saj D_2$ can be written in the block form
\[
L(D_1\saj D_2)=\left[\begin{array}{ccc}
rI_{n_1} & -B_{\textnormal{out}}(D_1) & O_{n_1,n_2} \\
-B_{\textnormal{in}}(D_1)^T & (1+n_2)I_{m_1} &  -J_{m_1,n_2} \\
O_{n_2,n_1} & -J_{n_2,m_1} & m_1I_{n_2}+L(D_2)
\end{array}\right],
\] 
and the signless Laplacian of $D_1\saj D_2$ can be written in the block form
\[
Q(D_1\saj D_2)=\left[\begin{array}{ccc}
rI_{n_1} & B_{\textnormal{out}}(D_1) & O_{n_1,n_2} \\
B_{\textnormal{in}}(D_1)^T & (1+n_2)I_{m_1} &  J_{m_1,n_2} \\
O_{n_2,n_1} & J_{n_2,m_1} & m_1I_{n_2}+Q(D_2)
\end{array}\right].
\]
Since $B_{\textnormal{in}}(D_1)^TB_{\textnormal{out}}(D_1)
=A(\mathcal L(D_1))$, the result now follows from \cref{cor:3by3}(2), using 
$\chi_{L(D_2)}(\lambda)=n_2/\lambda$ and \cref{prop:regular} to rewrite
$f_{A(\mathcal L(D_1))}$ in terms of $f_{L(D_1)}$ and $f_{Q(D_1)}$.
\end{proof}

\section{Subdivision-corona products}\label[section]{sec:subcor}
Let $G_1$ and $G_2$ be graphs where $G_1$ has $n_1$ vertices and $m_1$ edges.
The \emph{subdivision-vertex corona} $G_1\odot\, G_2$ is the graph obtained from $\mathcal{S}(G_1)$ and $n_1$ copies of $G_2$, all vertex-disjoint, by joining the $i$th vertex of $V(G_1)$ to every vertex in the $i$th copy of $G_2$, whereas the \emph{subdivision-edge corona} $G_1\ominus\, G_2$ is the graph obtained from $\mathcal{S}(G_1)$ and $m_1$ copies of $G_2$, all vertex-disjoint, by joining the $i$th vertex of $I(G_1)$ to every vertex in the $i$th copy of $G_2$.
The notation $G_2^{(v)}$ is used to denote the copy of $G_2$ that is joined to vertex $v$.
These two graph operations were introduced in \cite{PengliMiao2013} where their adjacency, Laplacian and signless Laplacian spectrum were computed when $G_1$ is regular and $G_2$ is arbitrary.
In \cite{LIU-subdivision}, the subdivision-vertex neighbourhood corona and subdivision-edge neighbourhood corona were defined and studied.
 
Before extending the subdivision-vertex and subdivision-edge coronas to the directed setting, we first establish a result for graphs that describes the $A$-spectrum of $G_1\odot\, G_2$ (resp. $G_1\ominus\, G_2$) when both factors $G_1$ and $G_2$ are arbitrary. 

\begin{theorem}
Let $G_1$ and $G_2$ be simple graphs on $n_1$ and $n_2$ vertices, and $m_1$ and $m_2$ edges, respectively. 
Then
\[
f_{A(G_1\odot\, G_2)}(\lambda)
=\left(\lambda-\chi_{A(G_2)}(\lambda)\right)^{n_1-m_1}\,\bigl[f_{A(G_2)}(\lambda)\bigr]^{n_1}\,f_{A(\mathcal{L}(G_1))}\left(
\lambda^2-\lambda\,\chi_{A(G_2)}(\lambda)-2
\right),
\]
and
\[
f_{A(G_1\ominus\, G_2)}(\lambda) 
=\lambda^{n_1-m_1}\,\bigl[f_{A(G_2)}(\lambda)\bigr]^{m_1}\,f_{A(\mathcal{L}(G_1))}\left(
\lambda^2-\lambda\,\chi_{A(G_2)}(\lambda)-2
\right),
\]
where $\mathcal{L}(G_1)$ is the line graph of $G_1$.
In particular, 
the $A$-spectrum of $G_1\odot\, G_2$ (resp. $G_1\ominus\, G_2$) is completely determined by the $A$-spectra of $\mathcal{L}(G_1)$ and $G_2$, and the $A$-coronal of $G_2$.
\end{theorem} 

\begin{proof}
Let $B(G_1)$ denote the incidence matrix of $G_1$ and suppose $G_2$ has $n_2$ vertices.
The adjacency matrix of $G_1\odot\, G_2$ is permutation similar to 
\[
A(G_1\odot\, G_2)=
\begin{bmatrix}
O_{n_1} & B(G_1) & I_{n_1}\otimes J_{1,n_2}\\
B(G_1)^{T} & O_{m_1} & O_{m_1,n_1n_2}\\
I_{n_1}\otimes J_{n_2,1} & O_{n_1n_2,m_1} & I_{n_1}\otimes A(G_2)
\end{bmatrix}
\]
with respect to the vertex partition $\bigl(V(G_1),I(G_1),\bigsqcup_{v\in V(G_1)} V(G_2^{(v)})\bigr)$.
Since $Q(G_1)=B(G_1)B(G_1)^T$, the first equation follows by \cref{cor:3by3}(iii) and \eqref{eq:QG}. 
Note that \eqref{eq:QG} gives $f_{Q(G_1)}(\lambda)=\lambda^{n_1-m_1}f_{A(\mathcal{L}(G_1))}(\lambda-2)$ and hence 
\[
f_{Q(G_1)}(\lambda^2-\lambda\chi_{A(G_2)}(\lambda))=\lambda^{n_1-m_1}
\left(\lambda-\chi_{A(G_2)}(\lambda)\right)^{n_1-m_1}
f_{A(\mathcal{L}(G_1))}(\lambda^2-\lambda\chi_{A(G_2)}(\lambda)-2).
\]
Similarly, the adjacency matrix of $G_1\ominus\, G_2$ is permutation similar to 
\[
A(G_1\ominus\, G_2)=
\begin{bmatrix}
O_{n_1} & B(G_1) & O_{n_1,m_1n_2}\\
B(G_1)^{T} & O_{m_1} & I_{m_1}\otimes J_{1,n_2}\\
O_{m_1n_2,n_1} & I_{m_1}\otimes J_{n_2,1} & I_{m_1}\otimes A(G_2)
\end{bmatrix}
\]
with respect to the vertex partition 
$\bigl(V(G_1),I(G_1),\bigsqcup_{v\in I(G_1)} V(G_2^{(v)})\bigr)$).
Since $A(\mathcal{L}(G_1))=B(G_1)^TB(G_1)-2I_{m_1}$, the second equation follows by \cref{cor:3by3}(iv) and \cref{cor_addI}.
\end{proof}
 
\subsection{The subdivision-vertex corona of two digraphs}
Let $D_1$ and $D_2$ be digraphs where $D_1$ has $n_1$ vertices.
The \emph{(symmetric) subdivision-vertex corona} $D_1\odot\, D_2$
is obtained 
by taking $\mathcal{S}(D_1)$ and $n_1$ copies of $D_2$ (one copy for each vertex $v\in V(D_1)$, denoted by $D_2^{(v)}$), and for every $v\in V(D_1)$, adding both arcs $vw$ and $wv$ for all  $w\in V(D_2^{(v)})$ (see \cref{fig:ssvc}(a)).
It is clear that $D_1\odot\, D_2$ is strongly connected if and only if $D_1$ is strongly connected.

\begin{figure}[ht]
\centering
\begin{subfigure}{0.45\textwidth}
\centering
\scalebox{0.95}{
\begin{tikzpicture}[>=Stealth, thick, scale=1, 
vertex/.style={circle, fill, inner sep=2pt},
ivertex/.style={circle, draw, inner sep=2pt},
midarrow/.style={line width=1.3pt, postaction={decorate, decoration={markings, mark=at position 0.599 with {\arrow{Stealth[length=3mm,width=3mm]}}}}}, 
doublearrow/.style={line width=2.0pt, red, dash pattern=on 1pt off 1pt}]
\usetikzlibrary{decorations.markings,arrows.meta}
\draw[blue, dashed] (-0.4,-0.4) rectangle (3.4,0.4)
                    (0.3,1.6) rectangle (2.7,2.4);
\foreach \x/\cx in {0/0, 1.5/1.5, 3/3} {
  \coordinate (d2c) at (\cx,-1.5);
  \node[draw, ellipse, line width=1.3pt, minimum width=1.2cm, minimum height=0.7cm] at (d2c) {$D_2$};
  \coordinate (el) at ($(d2c)+(-0.52,0.20)$);
  \coordinate (er) at ($(d2c)+(0.52,0.20)$);
  \draw[doublearrow] (\x,0) -- (el);
  \draw[doublearrow] (\x,0) -- (er);
}
\node[left] at (-0.5,0) {$V(D_1)$};
\node[left] at (0.2,2) {$I(D_1)$};
\node[vertex] (v1) at (0,0) {};
\node[vertex] (v2) at (1.5,0) {};
\node[vertex] (v3) at (3,0) {};
\node[ivertex] (i1) at (0.75,2) {};
\node[ivertex] (i2) at (2.25,2) {};
\foreach \a/\b in {v1/i1, i1/v2, v2/i2, i2/v3}
  \draw[midarrow] (\a) -- (\b);
\end{tikzpicture}}
\caption{The subdivision-vertex corona $D_1\odot\,D_2$}
\end{subfigure}
\hfill
\begin{subfigure}{0.45\textwidth}
\centering
\scalebox{0.95}{
\begin{tikzpicture}[>=Stealth, thick, scale=1, vertex/.style={circle, fill, inner sep=2pt}, 
ivertex/.style={circle, draw, inner sep=2pt},
midarrow/.style={line width=1.3pt, postaction={decorate, decoration={markings, mark=at position 0.599 with {\arrow{Stealth[length=3mm,width=3mm]}}}}}, 
doublearrow/.style={line width=2.0pt, red, dash pattern=on 1pt off 1pt}]
\usetikzlibrary{decorations.markings,arrows.meta}
\draw[blue, dashed] (-0.4,-0.4) rectangle (3.4,0.4)
                    (0.3,1.6) rectangle (2.7,2.4);
\foreach \ix/\cx in {0.75/0.75, 2.25/2.25} {
  \coordinate (d2c) at (\cx,3.5);
  \node[draw, ellipse, line width=1.3pt, minimum width=1.2cm, minimum height=0.7cm] at (d2c) {$D_2$};
  \coordinate (el) at ($(d2c)+(-0.52,-0.20)$);
  \coordinate (er) at ($(d2c)+(0.52,-0.20)$);
  \draw[doublearrow] (\ix,2.1) -- (el);
  \draw[doublearrow] (\ix,2.1) -- (er);
}                    
\node[left] at (0.2,2) {$I(D_1)$};
\node[left] at (-0.5,0) {$V(D_1)$};
\node[vertex] (v1) at (0,0) {};
\node[vertex] (v2) at (1.5,0) {};
\node[vertex] (v3) at (3,0) {};
\node[ivertex] (i1) at (0.75,2) {};
\node[ivertex] (i2) at (2.25,2) {};
\foreach \a/\b in {v1/i1, i1/v2, v2/i2, i2/v3}
  \draw[midarrow] (\a) -- (\b);
\end{tikzpicture}}
\caption{The subdivision-arc corona $D_1\ominus\,D_2$}
\end{subfigure}
\caption{
Above, $D_1=P_3$ (the directed path on three vertices) and $D_2$ is an arbitrary digraph. The dashed (red) lines represent all possible arcs in both directions between a vertex $v$ in $V(D_1)$ (for $D_1\odot\,D_2$) or $I(D_1)$ (for $D_1\ominus\,D_2$), and all vertices in the copy of $D_2$ corresponding to $v$.}
    \label{fig:ssvc}
\end{figure}
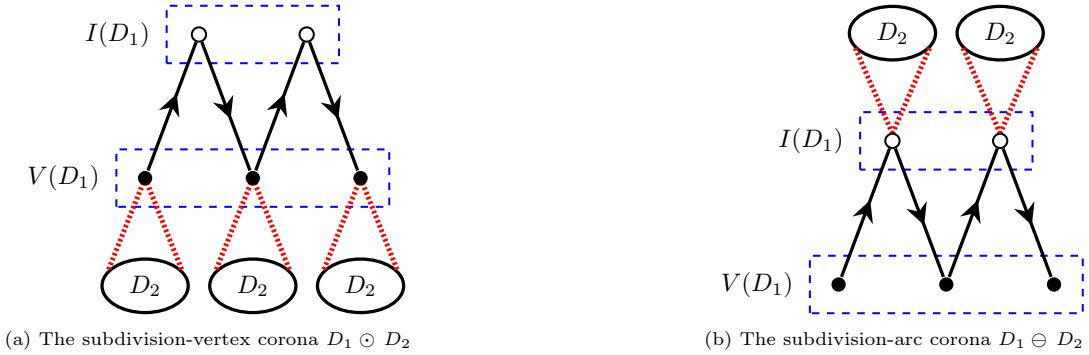

By ordering of the vertices of $D_1\odot\, D_2$ as $\bigl(V(D_1),I(D_1),\bigsqcup_{v\in V(D_1)} V(D_2^{(v)})\bigr)$, the adjacency, Laplacian and signless Laplacian matrices for $D_1\odot\, D_2$ are respectively permutation similar to
\[
A\left(D_1\odot\, D_2\right)=
\begin{bmatrix}
O_{n_1} & B_{\textnormal{out}}(D_1) & I_{n_1}\otimes J_{1,n_2}\\
B_{\textnormal{in}}(D_1)^{T} & O_{m_1} & O_{m_1,n_1n_2}\\
I_{n_1}\otimes J_{n_2,1} & O_{n_1n_2,m_1} & I_{n_1}\otimes A(D_2)
\end{bmatrix},
\] 
\[
L\bigl(D_1\odot\, D_2\bigr)
=
\begin{bmatrix}
D_{\textnormal{out}}(D_{1})+n_2I_{n_1} &
-\,B_{\textnormal{out}}(D_{1}) &
-\,I_{n_{1}}\otimes J_{1,n_{2}} \\
-\,B_{\textnormal{in}}(D_{1})^{T} &
I_{m_{1}} &
O_{m_1,n_1n_2} \\
-\,I_{n_{1}}\otimes J_{n_{2},1} &
O_{n_{1}n_{2},m_1} &
I_{n_{1}}\otimes \bigl(L(D_{2}) + I_{n_{2}}\bigr)
\end{bmatrix},
\]
\[
Q\bigl(D_1\odot\, D_2\bigr)
=
\begin{bmatrix}
D_{\textnormal{out}}(D_{1})+n_2I_{n_1} &
\,B_{\textnormal{out}}(D_{1}) &
\,I_{n_{1}}\otimes J_{1,n_{2}} \\
\,B_{\textnormal{in}}(D_{1})^{T} &
I_{m_{1}} &
O_{m_1,n_1n_2} \\
\,I_{n_{1}}\otimes J_{n_{2},1} &
O_{n_{1}n_{2},m_1} &
I_{n_{1}}\otimes \bigl(Q(D_{2})+I_{n_2}\bigr)
\end{bmatrix}.
\]  

\begin{theorem}
Let $D_1$ and $D_2$ be two digraphs on $n_1$ and $n_2$ vertices, and $m_1$ and $m_2$ arcs, respectively. Then
\[
f_{A\left(D_1\odot\, D_2\right)}(\lambda)=\lambda^{m_1-n_1}\left[f_{A\left(D_2\right)}(\lambda)\right]^{n_1} f_{A\left(D_1\right)}\left(\lambda^2-\lambda \chi_{A\left(D_2\right)}(\lambda)\right).
\]
In particular, the $A$-spectrum of $D_1\odot\, D_2$ is completely determined by the $A$-spectra of $D_1$ and $D_2$, and the $A$-coronal of $D_2$.
\end{theorem}

\begin{proof}
The result follows directly from \cref{lem:bb} and \cref{cor:3by3}(iii).
\end{proof}

In the case that $D_1$ is out-regular, we can obtain simple formulas for the Laplacian and signless Laplacian characteristic polynomial for $D_1\odot\, D_2$.

\begin{theorem}
Let $D_1$ and $D_2$ be two digraphs on $n_1$ and $n_2$ vertices, and $m_1$ and $m_2$ arcs, respectively.
Suppose that $D_1$ is $r$-out-regular.
Then 
\begin{enumerate}
    \item $\displaystyle f_{L(D_1\odot\, D_2)}(\lambda)=
    (-1)^{n_1}\left(\lambda-1\right)^{(r-1)n_1}\,
    f_{L(D_1)}\left(-\lambda^2+\lambda(n_2+r+1)\right)\,
    \bigl[f_{L(D_2)}(\lambda-1)\bigr]^{n_1}$,
    \item $\displaystyle f_{Q(D_1\odot\, D_2)}(\lambda)=
    \left(\lambda-1\right)^{(r-1)n_1}\,     
    f_{Q(D_1)}\left(r+(\lambda-1)(\lambda-r-n_2-\chi_{Q(D_2)}(\lambda-1))\right)\,
    \bigl[f_{Q(D_2)}(\lambda-1)\bigr]^{n_1}$.
\end{enumerate}    
In particular, the $L$-spectrum of $D_1\odot\, D_2$ is completely determined by the $L$-spectra of $D_1$ and $D_2$, and the $Q$-spectrum of $D_1\odot\, D_2$ is completely determined by the $Q$-spectra of $D_1$ and $D_2$, and the $Q$-coronal of $D_2$.
\end{theorem}

\begin{proof}
Since $D_1$ is $r$-out-regular, it has $m_1 = n_1 r$ arcs and $D_{\textnormal{out}}(D_1)=rI_{n_1}$.
For (1), by \cref{lem:bb} and \cref{cor:3by3}(iii) we have
\[ 
f_{L(D_1\odot\, D_2)}(\lambda)=
\left(\lambda-1\right)^{(r-1)n_1}\,
\bigl[f_{L(D_2)}(\lambda-1)\bigr]^{n_1}\,
f_{A(D_1)}\left((\lambda-1)\bigl(\lambda-r-n_2-\chi_{L(D_2)}(\lambda-1)\bigr)\right).
\]
Then \cref{prop:regular} gives
\[ 
f_{L(D_1\odot\, D_2)}(\lambda)=
\left(\lambda-1\right)^{(r-1)n_1}\,
\bigl[f_{L(D_2)}(\lambda-1)\bigr]^{n_1}\,
(-1)^{n_1}
f_{L(D_1)}\left(r-(\lambda-1)\bigl(\lambda-r-n_2-\chi_{L(D_2)}(\lambda-1)\bigr)\right),
\]
from which (1) follows since $\chi_{L(D_2)}(\lambda)=n_2/\lambda$ by \cref{coronal_cal}.

For (2), by \cref{lem:bb} and \cref{cor:3by3}(iii) we have
\[ 
f_{Q(D_1\odot\, D_2)}(\lambda)=
\left(\lambda-1\right)^{(r-1)n_1}\,
\bigl[f_{Q(D_2)}(\lambda-1)\bigr]^{n_1}\,
f_{A(D_1)}\left((\lambda-1)\bigl(\lambda-r-n_2-\chi_{Q(D_2)}(\lambda-1)\bigr)\right),
\]
from which (2) follows by applying \cref{prop:regular}.
\end{proof}

\subsection{The subdivision-arc corona of two digraphs}
Let $D_1$ and $D_2$ be digraphs where $D_1$ has $m_1$ arcs.
The \emph{(symmetric) subdivision-arc corona} $D_1\ominus D_2$
is obtained by taking $\mathcal{S}(D_1)$ and $m_1$ copies of $D_2$ (one copy for each vertex $v\in I(D_1)$, denoted by $D_2^{(v)}$), and for every $v\in I(D_1)$, adding both arcs $vw$ and $wv$ for all $w\in V(D_2^{(v)})$ (see \cref{fig:ssvc}(b)).
It is clear that $D_1\ominus\, D_2$ is strongly connected if and only if $D_1$ is strongly connected.

By ordering the vertices of $D_1\ominus\, D_2$ as $\bigl(V(D_1),I(D_1),\bigsqcup_{v\in I(D_1)} V(D_2^{(v)})\bigr)$, the adjacency, Laplacian and signless Laplacian matrices for $D_1\ominus\, D_2$ are respectively permutation similar to 
\[
A\left(D_1\ominus\, D_2\right)=
\begin{bmatrix}
O_{n_1} & B_{\textnormal{out}}(D_1) & O_{n_1,m_1n_2}\\
B_{\textnormal{in}}(D_1)^{T} & O_{m_1} & I_{m_1}\otimes J_{1,n_2}\\
O_{m_1n_2,n_1} & I_{m_1}\otimes J_{n_2,1} & I_{m_1}\otimes A(D_2)
\end{bmatrix},
\] 
\[
L\left(D_1\ominus\, D_2\right)=
\begin{bmatrix}
D_{\textnormal{out}}(D_{1}) 
& -\,B_{\textnormal{out}}(D_{1}) 
& O_{n_1,m_1n_2} \\
-\,B_{\textnormal{in}}(D_{1})^{T} 
& (\,1 + n_{2}\,)\,I_{m_{1}} 
& -\,I_{m_{1}}\;\otimes\;J_{1,n_{2}} \\
O_{m_1n_2,n_1}  
& -\,I_{m_{1}}\;\otimes\;J_{n_{2},1} 
& I_{m_{1}}\;\otimes \bigl(L(D_{2}) + I_{n_{2}}\bigr)
\end{bmatrix},
\]
\[
Q\left(D_1\ominus\, D_2\right)=
\begin{bmatrix}
D_{\textnormal{out}}(D_{1}) 
& \,B_{\textnormal{out}}(D_{1}) 
& O_{n_1,m_1n_2} \\
\,B_{\textnormal{in}}(D_{1})^{T} 
& (\,1 + n_{2}\,)\,I_{m_{1}} 
& \,I_{m_{1}}\;\otimes\;J_{1,n_{2}} \\
O_{m_1n_2,n_1} 
& \,I_{m_{1}}\;\otimes\;J_{n_{2},1} 
& I_{m_{1}}\;\otimes \bigl(Q(D_{2}) + I_{n_{2}}\bigr)
\end{bmatrix}.
\]

\begin{theorem}
Let $D_1$ and $D_2$ be digraphs on $n_1$ and $n_2$ vertices, and $m_1$ and $m_2$ arcs, respectively. Then
\[
f_{A\left(D_1\ominus\, D_2\right)}(\lambda)
=\left(\lambda-\chi_{A\left(D_2\right)}(\lambda)\right)^{m_1-n_1}\left[f_{A\left(D_2\right)}(\lambda)\right]^{m_1} 
f_{A(D_1)}\left(\lambda^2-\lambda \chi_{A\left(D_2\right)}(\lambda)\right).
\]
In particular, the $A$-spectrum of $D_1\ominus\, D_2$ is completely determined by the $A$-spectra of $D_1$ and $D_2$, and the $A$-coronal of $D_2$.
\end{theorem}

\begin{proof}
The result follows directly from \cref{cor:3by3}(iv)
and noting 
\[
f_{A(\mathcal{L}(D_1))}\left(\lambda^2-\lambda \chi_{A\left(D_2\right)}(\lambda)\right)
=\lambda^{m_1-n_1}\left(\lambda-\chi_{A\left(D_2\right)}(\lambda)\right)^{m_1-n_1}f_{A(D_1)}\left(\lambda^2-\lambda \chi_{A\left(D_2\right)}(\lambda)\right)
\]
by \cref{lem:bb}.
\end{proof}

In the case that $D_1$ is out-regular, we can obtain simple formulas for the Laplacian and signless Laplacian characteristic polynomial for $D_1\ominus\, D_2$.

\begin{theorem}
Let $D_1$ and $D_2$ be two digraphs on $n_1$ and $n_2$ vertices, and $m_1$ and $m_2$ arcs, respectively.
Suppose that $D_1$ is $r$-out-regular.
Then 
\begin{enumerate}[leftmargin=*]
    \item $\displaystyle f_{L(D_1\ominus\, D_2)}(\lambda)=$\\
    $\displaystyle (-1)^{n_1}\,\left(\lambda-1-n_2-n_2/(\lambda-1)\right)^{(r-1)n_1}\,
    f_{L(D_1)}\left(r-(\lambda-r)(\lambda-1-n_2-n_2/(\lambda-1))
    \right)\,
    \bigl[f_{L(D_2)}(\lambda-1)\bigr]^{n_1r}$.    
    \item $\displaystyle f_{Q(D_1\ominus\, D_2)}(\lambda)=$\\
    $\displaystyle\left(\lambda-1-n_2-\chi_{Q(D_2)}(\lambda-1)\right)^{(r-1)n_1}\,  
    f_{Q(D_1)}\left(r+(\lambda-r)(\lambda-1-n_2-\chi_{Q(D_2)}(\lambda-1))\right)\,
    \bigl[f_{Q(D_2)}(\lambda-1)\bigr]^{n_1r}$.
\end{enumerate}    
In particular, the $L$-spectrum of $D_1\ominus\, D_2$ is completely determined by the $L$-spectra of $D_1$ and $D_2$, and the $Q$-spectrum of $D_1\ominus\, D_2$ is completely determined by the $Q$-spectra of $D_1$ and $D_2$, and the $Q$-coronal of $D_2$.
\end{theorem}

\begin{proof}
Since $D_1$ is $r$-out-regular, it has $m_1 = n_1 r$ arcs and $D_{\textnormal{out}}(D_1)=rI_{n_1}$.
For (1), by \cref{lem:bb} and \cref{cor:3by3}(iv) we have
\[ 
f_{L(D_1\ominus\, D_2)}(\lambda)=
\left(\lambda-r\right)^{(1-r)n_1}\,
\bigl[f_{L(D_2)}(\lambda-1)\bigr]^{n_1r}\,
f_{A(\mathcal{L}(D_1))}\left((\lambda-r)(\lambda-1-n_2-\chi_{L(D_2)}(\lambda-1))\right).
\]
\cref{prop:regular} gives
\[ 
f_{A(\mathcal{L}(D_1))}(\lambda)
=\frac{(-1)^{n_1}}{\lambda^{n_1(1-r)}}f_{L(D_1)}(r-\lambda),
\]
from which (1) follows since $\chi_{L(D_2)}(\lambda)=n_2/\lambda$ by \cref{coronal_cal}.

For (2), by \cref{lem:bb} and \cref{cor:3by3}(iii) we have
\[ 
f_{Q(D_1\ominus\, D_2)}(\lambda)=
\left(\lambda-r\right)^{(1-r)n_1}\,
\bigl[f_{Q(D_2)}(\lambda-1)\bigr]^{n_1r}\,
f_{A(\mathcal{L}(D_1))}\left((\lambda-r)(\lambda-1-n_2-\chi_{Q(D_2)}(\lambda-1))\right).
\]
\cref{prop:regular} gives
\[ 
f_{A(\mathcal{L}(D_1))}(\lambda)
=\frac{1}{\lambda^{n_1(1-r)}}f_{Q(D_1)}(r+\lambda),
\]
from which (2) follows.
\end{proof}


\end{document}